\documentclass[11pt]{amsart}

\usepackage[margin=1.15in]{geometry}
\usepackage{amsmath,amssymb,amsfonts,amsthm,mathtools}
\usepackage{enumitem}
\usepackage{hyperref}
\usepackage{microtype}
\usepackage{mathrsfs}
\usepackage{bm}

\hypersetup{
  colorlinks=true,
  linkcolor=blue,
  citecolor=blue,
  urlcolor=blue,
  pdftitle={Sharp well-posedness and norm inflation for the periodic Camassa--Holm equation in critical Triebel--Lizorkin spaces},
  pdfauthor={Wenhai Shan and Xiao-Song Yang},
  pdfkeywords={Camassa--Holm equation, periodic domain, Triebel--Lizorkin spaces, norm inflation, Lagrangian coordinates, atomic decomposition}
}

\numberwithin{equation}{section}

\newtheorem{theorem}{Theorem}[section]
\newtheorem{proposition}[theorem]{Proposition}
\newtheorem{lemma}[theorem]{Lemma}

\theoremstyle{remark}
\newtheorem{remark}[theorem]{Remark}

\newcommand{\R}{\mathbb{R}}
\newcommand{\T}{\mathbb{T}}
\newcommand{\Z}{\mathbb{Z}}

\newcommand{\cN}{\mathcal{N}}

\newcommand{\cK}{\mathcal{K}}
\newcommand{\cP}{\mathcal{P}}
\newcommand{\eps}{\varepsilon}
\newcommand{\dd}{\,d}

\title[Sharp critical theory for periodic CH]{Sharp regularity for the periodic Camassa--Holm equation in critical Triebel--Lizorkin spaces}

\author{Wenhai Shan}
\address{School of Mathematics and Statistics, Huazhong University of Science and Technology, Wuhan, Hubei 430074, P.R. China}
\email{mathshanwenhai@163.com}

\author{Xiao-Song Yang}
\address{School of Mathematics and Statistics, Huazhong University of Science and Technology, Wuhan, Hubei 430074, P.R. China; 
	Hubei Key Laboratory of Engineering Modeling and Scientific Computing, Huazhong University of Science and Technology, Wuhan, Hubei 430074, P.R. China}
\email{yangxs@hust.edu.cn}
\thanks{Corresponding author: Xiao-Song Yang.}


\begin{document}

\begin{abstract}
We establish a sharp well-posedness and norm inflation theory for the Camassa--Holm equation in critical Triebel--Lizorkin $F^{1+1/p}_{p,q}(\mathbb{T})$. At the endpoint
$p=1$, we prove local Hadamard well-posedness for $1\le q<\infty$. In contrast, we prove norm inflation for $1<p<\infty$ and $1\le q\le\infty$. We also complement the
local well-posedness in the critical Besov spaces and higher-regularity Triebel--Lizorkin spaces. The positive results rely on a Lipschitz stability theorem for the periodic Green operator under degree-one Lagrangian flows. The negative result is based on a nested smooth atomic construction on the torus, adapted from its real-line counterpart.
\end{abstract}

\maketitle

\section{Introduction}

The Camassa--Holm equation
\begin{equation}\label{eq:CH-classical}
	u_t-u_{txx}+3uu_x=2u_xu_{xx}+uu_{xxx}
\end{equation}
was introduced as a model for unidirectional shallow-water waves
\cite{CamassaHolm1993}. It admits peakon solutions and exhibits the
characteristic phenomenon of wave breaking, whereby the solution itself
remains bounded while its spatial derivative becomes unbounded in finite
time. On the periodic domain, the classical local well-posedness theory
and the wave-breaking mechanism were developed in, among other works,
\cite{Constantin1997,ConstantinEscher1998}.

For the convenience of the analysis, equation \eqref{eq:CH-classical} can be written in the nonlocal form
\begin{equation}\label{eq:main-CH}
  \partial_tu+u\partial_xu=P_{\T}(D)\left(u^2+\frac12u_x^2\right),
  \qquad u|_{t=0}=u_0,
\end{equation}
where the nonlocal operator is the form of
\begin{equation}\label{eq:PD-def}
  P_{\T}(D)=-\partial_x(1-\partial_x^2)^{-1}_{\T}.
\end{equation}

The Besov-space theory for the Camassa--Holm equation was first
developed on the real line and torus by Danchin
\cite{Danchin2001,Danchin2003}. More precisely, local well-posedness were established in the inhomogeneous spaces
$
B^s_{p,r},1\le p\le\infty, 1\le r<\infty,s>\max\left\{1+\frac1p,\frac32\right\}.
$
On the
scaling-critical line \(s=1+1/p\), the feasible range was \(1\le p\le2\). The full critical Besov range $B^{1+1/p}_{p,1}(\R)$, $1\le p<\infty$, was later obtained by Ye, Yin and Guo \cite{YeYinGuo2023} through a Lagrangian argument. On the negative side, Guo, Liu, Molinet and Yin
\cite{GuoLiuMolinetYin2019} proved norm inflation for the
Camassa--Holm equation in $B^{1+1/p}_{p,r}(\mathbb K),
1\le p\le\infty, 1<r\le\infty,
\mathbb K\in\{\R,\T\}.
$ Their proof constructs smooth odd data with arbitrarily small critical Besov norm but an arbitrarily large negative derivative at the origin, and then combines the odd-data wave-breaking criterion with a logarithmic estimate. At the remaining periodic endpoint $B^1_{\infty,1}(\T)$, Li, Yu and Zhu \cite{LiYuZhu2025} constructed a different family of periodic initial data and proved norm inflation for the Camassa--Holm and Novikov equations.

The Triebel--Lizorkin setting is more delicate because the dyadic
summation is taken pointwise before the spatial norm. Yang \cite{Yang2015} established the local well-posedness in the inhomogeneous Triebel--Lizorkin spaces
$
F^s_{p,q}(\R),
1<p,q<\infty,
s>\max\left\{\frac32,1+\frac1p\right\}.
$ Very recently, Zhang and Yan developed transport and commutator estimates in Triebel--Lizorkin scales \cite{ZhangYan2026Transport,ZhangYan2026Commutator} and obtained a sharp real-line classification \cite{ZhangYan2026Sharp}: the equation is locally Hadamard well-posed in $F^2_{1,q}(\R)$ for $1\le q<\infty$, whereas it exhibits norm inflation in $F^{1+1/p}_{p,q}(\R)$ for $1<p<\infty$ and $1\le q\le\infty$. The key new ingredient in their ill-posedness proof is a smooth atomic construction adapted to the pointwise interaction of dyadic scales in the Triebel--Lizorkin space.

The purpose of the present paper is to develop a sharp periodic theory for the Camassa–Holm equation in critical Triebel–Lizorkin spaces and, at the same time, to fill the existing gap in the well-posedness theory for both Triebel–Lizorkin and Besov spaces on the torus. Although the overall strategy is inspired by the corresponding real-line theory, both the positive and negative arguments require genuinely periodic realizations: On the well-posedness side, the Green kernel of $(1-\partial_x^2)^{-1}_{\T}$ is the periodic sum
\begin{equation}\label{eq:periodic-kernel-intro}
	p_{\T}(x)=\sum_{m\in\Z}\frac12e^{-|x+m|},
\end{equation}
and one must prove Lipschitz stability of the associated nonlocal operator under degree-one Lagrangian flows. For norm inflation, one must verify that the real-line atomic mechanism can be realized by smooth periodic atoms with constants uniform in the dyadic scale. Once the initial data are constructed, the wave-breaking and logarithmic continuation parts follow the periodic framework of \cite{GuoLiuMolinetYin2019}.

Our first main result gives the well-posedness side of the sharp periodic classification in the Triebel–Lizorkin scale, together with the higher-regularity consequences and critical Besov.

\begin{theorem}[Local well-posedness]\label{thm:main}
Let $X^s(\T)$ be either
\begin{enumerate}[label=(\roman*)]
	\item
	\[
	X^s(\T)=B^{s}_{p,1}(\T),
	\qquad 1\le p<\infty,
	\qquad s=1+\frac1p;
	\]
	\item
	\[
	X^s(\T)=F^s_{p,q}(\T),
	\qquad 1\le p<\infty,
	\quad 1\le q<\infty,
	\quad s>1+\frac1p;
	\]\item
	\[
	X^s(\T)=F^s_{1,q}(\T),
	\quad 1\le q<\infty, s=2.
	\]
\end{enumerate}
For every $R>0$ there exists $T_R>0$ such that, whenever
$u_0\in X^s(\T)$ and $\|u_0\|_{X^s}\le R$, equation \eqref{eq:CH-classical}
has on $[0,T_R]$ a unique solution
\[
u\in C([0,T_R];X^s(\T))\cap C^1([0,T_R];X^{s-1}(\T)).
\]
The data-to-solution map is continuous from the ball
$\{u_0:\|u_0\|_{X^s}\le R\}$ into
\[
C([0,T_R];X^s(\T))\cap C^1([0,T_R];X^{s-1}(\T)).
\]
\end{theorem}

The second main result supplies the missing negative part on the torus.

\begin{theorem}[Norm inflation]\label{thm:norm-inflation}
Let
\[
  1<p<\infty,
  \qquad 1\le q\le\infty,
  \qquad s=1+\frac1p.
\]
For every $0<\varepsilon<1$, there exists a real-valued odd function
\[
  u_{0,\varepsilon}\in C^\infty(\T)
\]
such that
\begin{equation}\label{eq:small-initial-F}
  \|u_{0,\varepsilon}\|_{F^s_{p,q}(\T)}\le\varepsilon.
\end{equation}
The corresponding smooth solution has a finite maximal lifespan $T_\varepsilon$ satisfying
\begin{equation}\label{eq:short-lifespan}
  0<T_\varepsilon<\varepsilon,
\end{equation}
and
\begin{equation}\label{eq:norm-inflation-main}
  \limsup_{t\uparrow T_\varepsilon}
  \|u_\varepsilon(t)\|_{B^1_{\infty,\infty}(\T)}=\infty,
  \qquad
  \limsup_{t\uparrow T_\varepsilon}
  \|u_\varepsilon(t)\|_{F^s_{p,q}(\T)}=\infty.
\end{equation}
 
\end{theorem}
\begin{remark} 
	At higher regularity $s>1+\frac1p$, Theorem \ref{thm:main} improves the result of Yang \cite{Yang2015} by removing the additional restriction $s>3/2$, and the endpoint indices $p=1$ and $q=1$ are also included. At the critical regularity, combining Theorems \ref{thm:main} and \ref{thm:norm-inflation} with the periodic norm-inflation results of \cite{GuoLiuMolinetYin2019,LiYuZhu2025} yields a complete classification in the Besov scale: $B^{1+1/p}_{p,1}(\T)$ is locally Hadamard well-posed for $1\le p<\infty$, whereas norm inflation holds in $B^{1+1/p}_{p,r}(\T)$ for $1\le p\le\infty$ and $1<r\le\infty$, as well as at the remaining endpoint $B^1_{\infty,1}(\T)$. In the Triebel--Lizorkin scale, we obtain the sharp finite-$p$ classification
	\[
\begin{aligned}
	p=1,\quad 1\le q<\infty
	&:\quad \text{local Hadamard well-posedness},\\
	1<p<\infty,\quad 1\le q\le\infty
	&:\quad \text{norm inflation}.
\end{aligned}
\]
The remaining Triebel--Lizorkin endpoints are $F^2_{1,\infty}(\T)$ and $F^1_{\infty,q}(\T)$ with $1\le q<\infty$. The former is not covered because low-frequency truncations are not strongly convergent in the standard $F^2_{1,\infty}$ topology, while
the latter requires a separate analysis of the endpoint \(p=\infty\). The case $F^1_{\infty,\infty}(\T)=B^1_{\infty,\infty}(\T)$ is already covered by \cite{GuoLiuMolinetYin2019}.
\end{remark}

The paper is organized as follows. Section 2 recalls periodic function spaces, transport estimates and a periodic atomic synthesis estimate. Section 3 develops the periodic nonlocal operator in degree-one Lagrangian coordinates. Section 4 proves existence, uniqueness and continuous dependence. Section 5 establishes periodic norm inflation. Throughout this paper, we use $C$ to denote universal constant which may vary from line to line. 

\section{Function spaces and some estimates on torus}

Throughout Sections 2--4, $X^s$ denotes one of the spaces listed in
Theorem~\ref{thm:main}. Let $\chi\in C_c^\infty(\mathbb R)$ be an even function satisfying
\[
0\le \chi\le 1,\qquad
\chi(\xi)=1\quad\text{for }|\xi|\le \frac34,
\qquad
\operatorname{supp}\chi
\subset
\left\{\xi\in\mathbb R:|\xi|\le\frac43\right\}.
\]
Define
\[
\varphi(\xi):=\chi\left(\frac{\xi}{2}\right)-\chi(\xi).
\]
Then $\varphi\in C_c^\infty(\mathbb R)$ is even and
\[
\operatorname{supp}\varphi
\subset
\left\{\xi\in\mathbb R:
\frac34\le|\xi|\le\frac83\right\}.
\]
The pair $(\chi,\varphi)$ will be referred to as a standard
inhomogeneous dyadic partition of unity.

For a periodic distribution $f\in\mathcal D'(\mathbb T)$, where
$\mathbb T=\mathbb R/\mathbb Z$, let
\[
\widehat f(k)
=
\int_{\mathbb T}f(x)e^{-2\pi ikx}\,dx,
\qquad k\in\mathbb Z,
\]
we define
\begin{equation}\label{eq:periodic-blocks}
	\Delta_{-1}f
	=
	\sum_{k\in\mathbb Z}
	\chi(k)\widehat f(k)e^{2\pi ikx},
\end{equation}
and, for $j\ge0$,
\begin{equation}\label{eq:periodic-high-blocks}
	\Delta_jf
	=
	\sum_{k\in\mathbb Z}
	\varphi(2^{-j}k)\widehat f(k)e^{2\pi ikx}.
\end{equation}
 We also set
\[
S_jf
:=
\sum_{\ell=-1}^{j-1}\Delta_\ell f,
\qquad j\ge0.
\]
The corresponding periodic Besov and Triebel--Lizorkin norms are
\begin{equation}\label{eq:Besov-norm}
	\|f\|_{B^s_{p,r}(\mathbb T)}
	=
	\left\|
	\left(
	2^{js}\|\Delta_jf\|_{L^p(\mathbb T)}
	\right)_{j\ge-1}
	\right\|_{\ell^r},
\end{equation}
and
\begin{equation}\label{eq:F-norm}
	\|f\|_{F^s_{p,q}(\mathbb T)}
	=
	\left\|
	\left(
	\sum_{j\ge-1}
	2^{jsq}|\Delta_jf(x)|^q
	\right)^{1/q}
	\right\|_{L^p_x(\mathbb T)},
\end{equation}
with the usual supremum modifications when $r=\infty$ or
$q=\infty$, respectively.  

We shall use the following embeddings and Moser estimates, see \cite{GuoLi2021,RunstSickel1996,SickelTriebel1995}. 
\begin{lemma}
	The spaces occurring in Theorem \ref{thm:main} satisfy
	\begin{equation}\label{eq:X-embedding}
		X^s(\T)\hookrightarrow W^{1,\infty}(\T),
		\qquad X^{s-1}(\T)\hookrightarrow L^\infty(\T).
	\end{equation}
In addition, for $\sigma>0$ and $f,g\in X^\sigma(\T)\cap L^\infty(\T)$, one has
	\begin{equation}\label{eq:Moser}
		\|fg\|_{X^\sigma}\le C\big(\|f\|_{L^\infty}\|g\|_{X^\sigma}
		+\|g\|_{L^\infty}\|f\|_{X^\sigma}\big).
	\end{equation}
Therefore, $X^{s-1}(\T)$ is a Banach algebra.
\end{lemma}

We shall use the following transport estimate. In the Besov scale it is classical; see \cite{Danchin2005,BCD2011}. In the Triebel--Lizorkin scale it is the one-dimensional periodic counterpart of the commutator estimate and transport theorems in \cite[Theorems 1.1--1.3]{ZhangYan2026Commutator}; see also \cite{ZhangYan2026Transport}. 

\begin{proposition}[Periodic transport estimate]\label{prop:transport}
Let $\sigma\in\{s,s-1\}$. Suppose that
\begin{equation}\label{eq:transport-eqn}
  \partial_t f+a\partial_x f=g,
  \qquad f(0)=f_0,
\end{equation}
where
\[
  a\in L^1(0,T;X^\sigma(\T)),\qquad
  f_0\in X^\sigma(\T),\qquad
  g\in L^1(0,T;X^\sigma(\T)).
\]
Then
\begin{equation}\label{eq:transport-est}
  \|f(t)\|_{X^\sigma}
  \le C\exp\left(C\int_0^t\|a(\tau)\|_{X^\sigma}\dd\tau\right)
  \left(\|f_0\|_{X^\sigma}+\int_0^t\|g(\tau)\|_{X^\sigma}\dd\tau\right).
\end{equation}
Moreover, $f\in C([0,T];X^\sigma(\T))$.
\end{proposition}

The following assertions are standard consequences of the periodic Littlewood--Paley characterization, see \cite{BCD2011,Triebel1983,Triebel1992}.
\begin{lemma} \label{lem:low-frequency} For any $\ell\ge0$ and integer $N\geq 0$,
\begin{equation}\label{eq:SN-smoothing}
  \|S_N f\|_{X^{s+\ell}}\le C2^{N\ell}\|f\|_{X^s}.
\end{equation}
Moreover,
\begin{equation}\label{eq:SN-convergence}
  S_Nf\to f\quad\text{in }X^s(\T),
\end{equation}
and for any $\eps>0$,
\begin{equation}\label{eq:compact-embedding}
  X^s(\T)\hookrightarrow\hookrightarrow X^{s-\eps}(\T).
\end{equation}
\end{lemma}
We also need a coefficient-stability result in the proof of continuous dependence.

\begin{lemma}[Transport stability]\label{lem:transport-stability}
	Let
	\[
	\partial_ta^n+A^n\partial_xa^n=f,\qquad a^n(0)=a_0,
	\]
	where \(a_0\in X^{s-1}\), \(f\in L^1(0,T;X^{s-1})\),
	\[
	\sup_n\|A^n(t)\|_{X^s}\le\beta(t),\quad \beta\in L^1(0,T),
	\qquad
	A^n\to A^\infty\quad\hbox{in }L^1(0,T;X^{s-1}).
	\]
	Then
	\[
	a^n\to a^\infty\quad\hbox{in }C([0,T];X^{s-1}).
	\]
\end{lemma}

\begin{proof}
	Let \(a_N^n\) solve the same transport equation with data \(S_Na_0\) and
	source \(S_Nf\). Proposition~\ref{prop:transport} gives, uniformly in
	\(n\in\mathbb N\cup\{\infty\}\),
	\begin{equation}\label{eq:stability-tail}
		\|a^n-a_N^n\|_{C_TX^{s-1}}
		\le C_\beta\bigl(
		\|(I-S_N)a_0\|_{X^{s-1}}
		+\|(I-S_N)f\|_{L^1_TX^{s-1}}\bigr)\longrightarrow0.
	\end{equation}
	Here one uses strong low-frequency approximation (\ref{eq:SN-convergence}) and dominated convergence in time.
	
	For fixed \(N\), \(a_N^\infty\in L^\infty(0,T;X^s)\), and
	\(b_N^n=a_N^n-a_N^\infty\) satisfies
	\[
	\partial_tb_N^n+A^n\partial_xb_N^n
	=(A^\infty-A^n)\partial_xa_N^\infty,\qquad b_N^n(0)=0.
	\]
	Since \(X^{s-1}\) is an algebra, another application of the transport estimate
	yields
	\begin{equation}\label{eq:stability-fixed-N}
		\|a_N^n-a_N^\infty\|_{C_TX^{s-1}}
		\le C_N\|A^n-A^\infty\|_{L^1_TX^{s-1}}\longrightarrow0.
	\end{equation}
	The decomposition
	\[
	a^n-a^\infty=(a^n-a_N^n)+(a_N^n-a_N^\infty)
	+(a_N^\infty-a^\infty)
	\]
	and the order of limits \(n\to\infty\), then \(N\to\infty\), finish the proof.
\end{proof}

To prove the norm inflation on torus in Triebel--Lizorkin. The only genuinely periodic point in the atomic construction is that localization and periodization must preserve the Triebel--Lizorkin norm with constants independent of the dyadic scale. We first prove this fact. Let $\pi:\R\to\T$ be the quotient map and, for a function supported in a fundamental interval, define
\[
  (\mathcal Pf)(\pi(x))=\sum_{m\in\Z}f(x+m).
\]

\begin{proposition}[Uniform local periodization]\label{prop:local-periodization}
Let $I\Subset(-1/2,1/2)$, $1\le p<\infty$, $1\le q\le\infty$, and $s\in \R$. There exists a constant $C=C(I,p,q)$ such that, for every $f\in F^s_{p,q}(\R)$ with $\operatorname{supp}f\subset I$,
\begin{equation}\label{eq:periodization-equivalence}
  C^{-1}\|f\|_{F^s_{p,q}(\R)}
  \le \|\mathcal Pf\|_{F^s_{p,q}(\T)}
  \le C\|f\|_{F^s_{p,q}(\R)}.
\end{equation}
In particular, the same constant applies to arbitrary dyadic rescalings whose supports remain in $I$.
\end{proposition}

\begin{proof}
Choose $\chi \in C_c^\infty((-1/2,1/2))$ such that $\chi\equiv1$ on $I$. Since $\mathcal Pf$ is supported in  $\pi((-1/2,1/2))$, this characterization gives
\[
  \|\mathcal Pf\|_{F^s_{p,q}(\T)}
  \lesssim \|\chi f\|_{F^s_{p,q}(\R)}
  \lesssim \|f\|_{F^s_{p,q}(\R)}.
\]
Conversely, in the chosen coordinate interval,
\[
  f=\chi\big((\mathcal Pf)\circ\pi\big).
\]
The multiplier property of $\chi$ therefore yield
\[
  \|f\|_{F^s_{p,q}(\R)}
  \lesssim \|\mathcal Pf\|_{F^s_{p,q}(\T)}.
\]
This proves the scale-uniform assertion.
\end{proof}
For $x,y\in\T $, let
\[
\operatorname{dist}_{\T}(x,y)
:=
\min_{m\in\Z}|x-y+m|.
\]
We now record the precise synthesis estimate needed below. It follows from Proposition \ref{prop:local-periodization} and the Euclidean smooth atomic synthesis theorem (see \cite{Skrzypczak1998,ZhangYan2026Sharp} and \cite[Section~3.2]{Triebel1992}).

\begin{lemma}[Nested periodic atoms]\label{lem:periodic-atoms}
Let $1\le p<\infty$, $1\le q\le\infty$, and $s=1+1/p$. Fix $c\in(0,1/4)$ and set
\[
  Q_k=\{x\in\T:\operatorname{dist}_{\T}(x,0)\le c2^{-k}\},
  \qquad k\ge 0.
\]
Suppose that $a_k\in C^\infty(\T)$ satisfies
\begin{equation}\label{eq:atom-support}
  \operatorname{supp}a_k\subset Q_k
\end{equation}
and, for some integer $K>s$,
\begin{equation}\label{eq:atom-derivatives}
  \|\partial_x^\ell a_k\|_{L^\infty}
  \le C_\ell 2^{-k+k\ell},
  \qquad 0\le\ell\le K,
\end{equation}
with constants independent of $k$. Then, for every finite scalar sequence $(\lambda_k)$,
\begin{equation}\label{eq:atomic-synthesis}
  \left\|\sum_{k=0}^N\lambda_ka_k\right\|_{F^s_{p,q}(\T)}
  \le C
  \left\|
  \left(\sum_{k=0}^N
  |\lambda_k2^{k/p}\mathbf1_{Q_k}|^q\right)^{1/q}
  \right\|_{L^p(\T)},
\end{equation}
with the usual supremum modification for $q=\infty$.
\end{lemma}

\begin{proof}
	Let
	\[
	I_0:=\left(-\frac12,\frac12\right),
	\qquad
	U:=\pi(I_0)\subset\T.
	\]
For each $k\ge0$, let $\widetilde Q_k\subset I_0$ be the unique lift
	of $Q_k$ under $\pi|_{I_0}$, and define
	\[
	\widetilde a_k(x)
	:=
	\begin{cases}
		a_k(\pi(x)),&x\in I_0,\\
		0,&x\notin I_0.
	\end{cases}
	\]
	Because $\operatorname{supp}a_k\subset Q_k\Subset U$, the function
	$a_k\circ\pi$ vanishes in a neighborhood of the endpoints
	$\pm1/2$. Hence its extension by zero is smooth, and therefore
	\[
	\widetilde a_k\in C_c^\infty(\R),
	\qquad
	\operatorname{supp}\widetilde a_k
	\subset\widetilde Q_k,
	\qquad
	|\widetilde Q_k|\asymp 2^{-k}.
	\]
	Since the coordinate map is the identity in the chosen coordinate
	variable, the derivative assumptions on $a_k$ give
	\[
	\|\partial_x^\ell\widetilde a_k\|_{L^\infty(\R)}
	\le C_\ell 2^{-k+k\ell},
	\qquad 0\le\ell\le K.
	\]
	Let $\widetilde g_N=\sum_{k=0}^N\lambda_k\widetilde a_k$, the Euclidean smooth atomic synthesis theorem yields
	\begin{align}
		\|\widetilde g_N\|_{F^s_{p,q}(\R)}
		\lesssim
		\left\|
		\left(
		\sum_{k=0}^N
		\left|
		\lambda_k
		|\widetilde Q_k|^{-1/p}
		\mathbf1_{\widetilde Q_k}
		\right|^q
		\right)^{1/q}
		\right\|_{L^p(\R)}                                                    
		\lesssim
		\left\|
		\left(
		\sum_{k=0}^N
		\left|
		\lambda_k2^{k/p}\mathbf1_{\widetilde Q_k}
		\right|^q
		\right)^{1/q}
		\right\|_{L^p(\R)}.
		\label{eq:Euclidean-atomic-synthesis}
	\end{align}
	Since the intervals are nested,
	\[
	\widetilde Q_k\subset\widetilde Q_0,
	\qquad 0\le k\le N,
	\]
	and hence
	\[
	\operatorname{supp}\widetilde g_N
	\subset\widetilde Q_0
	\Subset I_0.
	\]
	Thus Proposition~\ref{prop:local-periodization} can be applied to the
	whole finite sum $\widetilde g_N$, thus obtaining the estimate. 
\end{proof}

\section{some estimates for periodic nonlocal operator}

The nonlocal operator $P_{\T}(D)$ is of order $-1$. Hence the smooth Fourier multiplier theorem in Besov and Triebel--Lizorkin spaces yields the following lifting by one derivative; see \cite{RunstSickel1996,Triebel1983}.

\begin{lemma}[Multiplier estimate]\label{lem:multiplier}
For all spaces occurring in Theorem \ref{thm:main},
\begin{equation}\label{eq:multiplier-est}
  \|P_{\T}(D)f\|_{X^s}\le C\|f\|_{X^{s-1}}.
\end{equation}
\end{lemma}

Define
\begin{equation}\label{eq:N-def}
  \cN_{\T}(u)=P_{\T}(D)\left(u^2+\frac12u_x^2\right),
\end{equation}
we have the following Nonlinear estimates:
\begin{lemma}\label{lem:nonlinear}
For $u\in X^s(\T)$,
\begin{equation}\label{eq:N-bound}
  \|\cN_{\T}(u)\|_{X^s}\le C\|u\|_{X^s}^2.
\end{equation}
If $u,v$ belong to a ball of radius $R$ in $X^s(\T)$, then
\begin{equation}\label{eq:N-Lipschitz}
  \|\cN_{\T}(u)-\cN_{\T}(v)\|_{X^s}\le C_R\|u-v\|_{X^s}.
\end{equation}
\end{lemma}

\begin{proof}
By Lemma \ref{lem:multiplier}, the embedding \eqref{eq:X-embedding}, and the algebra estimate \eqref{eq:Moser},
\[
  \|\cN_{\T}(u)\|_{X^s}
  \lesssim \left\|u^2+\frac12u_x^2\right\|_{X^{s-1}}
  \lesssim \|u\|_{X^s}^2.
\]
The difference estimate follows from the same bound after the polarization identities
\[
  u^2-v^2=(u-v)(u+v),\qquad
  u_x^2-v_x^2=(u_x-v_x)(u_x+v_x).
\]
\end{proof}

The periodic Green kernel of $(1-\partial_x^2)^{-1}_{\T}$ is
\begin{equation}\label{eq:pT-def}
  p_{\T}(x)=\sum_{m\in\Z}\frac12 e^{-|x+m|}.
\end{equation}
For $x\in[0,1)$,
\begin{equation}\label{eq:pT-cosh}
  p_{\T}(x)=\frac{\cosh(x-1/2)}{2\sinh(1/2)}.
\end{equation}
Equivalently, for all $x\in\R$,
\begin{equation}\label{eq:pT-cosh-global}
  p_{\T}(x)=\frac{\cosh(x-\lfloor x\rfloor-1/2)}{2\sinh(1/2)}.
\end{equation}
Let
\begin{equation}\label{eq:KT-def}
  K_{\T}(x)=-p_{\T}'(x).
\end{equation}
Then
\begin{equation}\label{eq:p-K-conv}
  (1-\partial_x^2)^{-1}_{\T}f=p_{\T}*f,
  \qquad P_{\T}(D)f=K_{\T}*f.
\end{equation}

Let $u$ be periodic and let $y(t,\xi)$ be the degree-one lift of the Lagrangian flow,
\begin{equation}\label{eq:Lagrange-flow}
  \frac{d}{dt}y(t,\xi)=u(t,y(t,\xi)),
  \qquad y(0,\xi)=\xi,
\end{equation}
with
\begin{equation}\label{eq:degree-one}
  y(t,\xi+1)=y(t,\xi)+1.
\end{equation}
Set
\begin{equation}\label{eq:U-def}
  U(t,\xi)=u(t,y(t,\xi)).
\end{equation}
If $u_x\in L^1(0,T;L^\infty)$, then
\begin{equation}\label{eq:y-xi-formula}
  y_\xi(t,\xi)=\exp\left(\int_0^t u_x(\tau,y(\tau,\xi))\dd\tau\right)>0.
\end{equation}
In particular, take $T$ small enough, there are constants $0<c_0<C_0$ such that
\begin{equation}\label{eq:yxi-bounds}
  0<c_0\le y_\xi(t,\xi)\le C_0.
\end{equation}

The map $y$ itself is not periodic. Rather, $y(t,\xi)-\xi$ is periodic. Likewise, $U$ and $y_\xi$ are periodic. The lift is used only to keep the order structure of the real line.
 
 For $c>0$ define
\begin{equation}\label{eq:Ec-def}
  E_c(z)=\sum_{m\in\Z}e^{-c|z+m|}.
\end{equation}
Then we have the following periodic estimates.

\begin{lemma}[Kernel perturbation]\label{lem:kernel-perturbation}
Let $y_i$($i=1,2$) be the degree-one lift of the Lagrangian flow defined in (\ref{eq:Lagrange-flow}). Set $Y=y_1-y_2$. Then for some $C>0$ and $\xi,\eta\in\R$,
\begin{align}
&|p_{\T}(y_1(\xi)-y_1(\eta))-p_{\T}(y_2(\xi)-y_2(\eta))| \notag\\
&\hspace{3cm}\le C E_c(\xi-\eta)(|Y(\xi)|+|Y(\eta)|),\label{eq:p-kernel-perturb}
\end{align}
and
\begin{align}
&|K_{\T}(y_1(\xi)-y_1(\eta))-K_{\T}(y_2(\xi)-y_2(\eta))| \notag\\
&\hspace{3cm}\le C E_c(\xi-\eta)(|Y(\xi)|+|Y(\eta)|).\label{eq:K-kernel-perturb}
\end{align}
\end{lemma}

\begin{proof}
We prove \eqref{eq:p-kernel-perturb}. By \eqref{eq:pT-def}, set
\[
  a_m=y_1(\xi)-y_1(\eta)+m,
  \qquad b_m=y_2(\xi)-y_2(\eta)+m.
\]
Using the degree-one property,
\[
  a_m=y_1(\xi+m)-y_1(\eta),
  \qquad b_m=y_2(\xi+m)-y_2(\eta).
\]
Since $(y_i)_\xi\ge c_0$,
\[
  |a_m|\ge c_0|\xi-\eta+m|,
  \qquad |b_m|\ge c_0|\xi-\eta+m|.
\]
Moreover, strict monotonicity gives
\[
  \operatorname{sign}(a_m)=\operatorname{sign}(b_m)
  =\operatorname{sign}(\xi-\eta+m),
\]
except when $\xi-\eta+m=0$, in which case $a_m=b_m=0$. Hence every point on the
line segment joining $a_m$ and $b_m$ has absolute value at least
$c_0|\xi-\eta+m|$. Since
\[
  |a_m-b_m|=|Y(\xi)-Y(\eta)|\le |Y(\xi)|+|Y(\eta)|,
\]
the mean value theorem yields
\[
  |e^{-|a_m|}-e^{-|b_m|}|
  \le C e^{-c|\xi-\eta+m|}(|Y(\xi)|+|Y(\eta)|).
\]
Summing in $m$ gives \eqref{eq:p-kernel-perturb}.

For $K_{\T}$, write 
\[
K_{\T}(z)
=
\sum_{m\in\Z}K_{\R}(z+m),
\qquad
K_{\R}(z)
=
\frac12\operatorname{sign}(z)e^{-|z|}.
\]
Using the degree-one property, the sign factors in
$K_{\R}$ coincide, and
\begin{align*}
	&\left|
	K_{\R}\bigl(y_1(\xi)-y_1(\eta)+m\bigr)
	-
	K_{\R}\bigl(y_2(\xi)-y_2(\eta)+m\bigr)
	\right|
	\\
	&\qquad=
	\frac12
	\left|
	e^{-|y_1(\xi)-y_1(\eta)+m|}
	-
	e^{-|y_2(\xi)-y_2(\eta)+m|}
	\right|.
\end{align*}
Thus by the same mean-value estimate used above and summing over $m\in\Z$ gives the desired
estimate for $K_{\T}$.
\end{proof}
For a degree-one lift $y$ and \(U(\xi)=u(y(\xi))\), by the change of variables over one
period, we set $	A=U^2y_\eta+\frac12\frac{U_\eta^2}{y_\eta}$ and
\begin{equation}\label{eq:Ky-def}
  \cK_y A(\xi)=\int_{\T}K_{\T}(y(\xi)-y(\eta))A(\eta)\dd\eta,
\end{equation}
\begin{equation}\label{eq:Py-def}
  \cP_y A(\xi)=\int_{\T}p_{\T}(y(\xi)-y(\eta))A(\eta)\dd\eta.
\end{equation}

\begin{lemma}[Lagrangian Lipschitz estimate]\label{lem:Lagrange-Lipschitz}
Let $1\le r\le\infty$. Write $y_i(\xi)=\xi+Y_i(\xi)$ and assume that $U_i,Y_i$ are uniformly bounded in $W^{1,\infty}(\T)\cap W^{1,r}(\T)$. Then
\begin{equation}\label{eq:K-Lipschitz-Lr}
  \|\cK_{y_1}A_1-\cK_{y_2}A_2\|_{L^r}
  \le C\big(\|U_1-U_2\|_{W^{1,r}}+\|Y_1-Y_2\|_{W^{1,r}}\big),
\end{equation}
and
\begin{equation}\label{eq:K-Lipschitz-W1r}
  \|\cK_{y_1}A_1-\cK_{y_2}A_2\|_{W^{1,r}}
  \le C\big(\|U_1-U_2\|_{W^{1,r}}+\|Y_1-Y_2\|_{W^{1,r}}\big).
\end{equation}
\end{lemma}

\begin{proof}
First,
\begin{equation}\label{eq:A-diff-bound}
  \|A_1-A_2\|_{L^r}
  \le C\big(\|U_1-U_2\|_{W^{1,r}}+\|y_1-y_2\|_{W^{1,r}}\big).
\end{equation}
This follows by expanding
\[
  U_1^2(y_1)_\xi-U_2^2(y_2)_\xi
  =(U_1-U_2)(U_1+U_2)(y_1)_\xi+U_2^2((y_1)_\xi-(y_2)_\xi),
\]
and
\[
 \frac{(U_1)_\xi^2}{(y_1)_\xi}-\frac{(U_2)_\xi^2}{(y_2)_\xi}
 =\frac{((U_1)_\xi-(U_2)_\xi)((U_1)_\xi+(U_2)_\xi)}{(y_1)_\xi}
 +(U_2)_\xi^2\frac{(y_2)_\xi-(y_1)_\xi}{(y_1)_\xi (y_2)_\xi}.
\]

Next decompose
\[
  \cK_{y_1}A_1-\cK_{y_2}A_2
  =\cK_{y_1}(A_1-A_2)+(\cK_{y_1}-\cK_{y_2})A_2.
\]
Since $|K_{\T}(y_1(\xi)-y_1(\eta))|\le CE_c(\xi-\eta)$, Young's inequality on $\T$ yields
\[
  \|\cK_{y_1}(A_1-A_2)\|_{L^r}\le C\|A_1-A_2\|_{L^r}.
\]
For the kernel difference, Lemma \ref{lem:kernel-perturbation} gives
\[
 |(\cK_{y_1}-\cK_{y_2})A_2(\xi)|
 \le C|Y(\xi)|(E_c*|A_2|)(\xi)+C E_c*(|Y||A_2|)(\xi),
\]
where $Y=Y_1-Y_2=y_1-y_2$. Since $A_2$ is uniformly bounded in $L^\infty\cap L^r$,
\[
  \|(\cK_{y_1}-\cK_{y_2})A_2\|_{L^r}\le C\|Y\|_{L^r}.
\]
This proves \eqref{eq:K-Lipschitz-Lr}.

It remains to estimate the derivative. Define
\[
  q_y(x)=\frac{A(y^{-1}(x))}{y_\xi(y^{-1}(x))}.
\]
Then
\[
  \cK_yA(\xi)=P_{\T}(D)q_y(y(\xi)).
\]
The distributional identity
\[
  \partial_xP_{\T}(D)q_y
  =q_y-p_{\T}*q_y
\]
follows from $(1-\partial_x^2)p_{\T}=\delta_{\T}$. Hence the chain rule gives
\begin{align*}
  \partial_\xi\cK_yA(\xi)
  &=y_\xi(\xi)\big(q_y(y(\xi))-(p_{\T}*q_y)(y(\xi))\big)\\
  &=A(\xi)-y_\xi(\xi)\cP_yA(\xi).
\end{align*}
Thus
\begin{equation}\label{eq:derivative-identity}
  \partial_\xi \cK_yA=A-y_\xi\cP_yA.
\end{equation}
Therefore
\[
  \partial_\xi(\cK_{y_1}A_1-\cK_{y_2}A_2)
  =(A_1-A_2)-\big((y_1)_\xi\cP_{y_1}A_1-(y_2)_\xi\cP_{y_2}A_2\big).
\]
We split
\begin{align*}
 &(y_1)_\xi\cP_{y_1}A_1-(y_2)_\xi\cP_{y_2}A_2\\
 &\quad=((y_1)_\xi-(y_2)_\xi)\cP_{y_1}A_1
 +(y_2)_\xi\cP_{y_1}(A_1-A_2)
 +(y_2)_\xi(\cP_{y_1}-\cP_{y_2})A_2.
\end{align*}
The estimates for $\cP_y$ are the same as those for $\cK_y$, using \eqref{eq:p-kernel-perturb}. Combining these bounds with \eqref{eq:A-diff-bound} gives \eqref{eq:K-Lipschitz-W1r}.
\end{proof}

\section{Well-posedness}
 Throughout this section, $p$ denotes the spatial integrability
 index of the chosen space $X^s$; in particular, $p=1$ when
 $X^s=F^2_{1,q}(\T)$.
\subsection{Existence}
\leavevmode\par
\medskip
Let $u_0\in X^s(\T)$ and set $u_0^n=S_nu_0$. By the classical smooth periodic
Camassa--Holm theory \cite{GuoLiuMolinetYin2019,RodriguezBlanco2001}, there is a
smooth solution $u^n$ on an interval $[0,T_n)$ with initial value $u_0^n$. We first
obtain a lifespan and bounds independent of $n$. Applying Proposition
\ref{prop:transport} to the equation satisfied by $u^n$ and using Lemma
\ref{lem:nonlinear}, we obtain, for $t<T_n$,
\begin{equation}\label{eq:smooth-approx-bound}
  \|u^n(t)\|_{X^s}
  \le e^{C\int_0^t\|u^n(\tau)\|_{X^s}\dd\tau}
  \left(\|u_0^n\|_{X^s}
  +C\int_0^t\|u^n(\tau)\|_{X^s}^2\dd\tau\right).
\end{equation}
Since $\|u_0^n\|_{X^s}\le C\|u_0\|_{X^s}$, a standard continuity argument gives a
time $T=T(\|u_0\|_{X^s})>0$ such that
\begin{equation}\label{eq:uniform-Xs}
  \sup_n\|u^n\|_{L^\infty(0,\min\{T,T_n\};X^s)}
  \le C\|u_0\|_{X^s}.
\end{equation}
The bound \eqref{eq:uniform-Xs} holds on the common interval $[0,T]$.
Indeed, smooth solutions obey, for every integer $m\ge3$,
\begin{equation}\label{eq:smooth-Hm-continuation}
  \frac{d}{dt}\|u^n(t)\|_{H^m}
  \le C\bigl(\|u^n(t)\|_{L^\infty}
  +\|u_x^n(t)\|_{L^\infty}\bigr)\|u^n(t)\|_{H^m}.
\end{equation}
If $T_n<T$, then \eqref{eq:uniform-Xs} and
$X^s\hookrightarrow W^{1,\infty}$ make the coefficient in
\eqref{eq:smooth-Hm-continuation} integrable on $[0,T_n)$, so all $H^m$
norms remain finite and the classical smooth theory
\cite{GuoLiuMolinetYin2019,RodriguezBlanco2001} extends $u^n$ beyond $T_n$, a
contradiction. Hence $T_n\ge T$. The equation (\ref{eq:main-CH}) and Lemma~\ref{lem:nonlinear}
also give
\begin{equation}\label{eq:uniform-time-derivative}
  \sup_n\|\partial_tu^n\|_{L^\infty(0,T;X^{s-1})}\le C.
\end{equation}

Consequently the sequence is Lipschitz in time with values in $X^{s-1}$ and, by
interpolation, uniformly H\"older continuous with values in $X^{s-\eps}$. Fix
\[
  0<\eps<\min\{1,s-1\};
\]
then $X^{s-\eps}(\T)\hookrightarrow W^{1,p}(\T)$. (\ref{eq:compact-embedding}) in Lemma
\ref{lem:low-frequency} and Arzel\`a--Ascoli theorem give, after extraction,
\begin{equation}\label{eq:compact-conv}
  u^n\to u\quad\text{in }C([0,T];X^{s-\eps}(\T))
  \hookrightarrow C([0,T];W^{1,p}(\T)).
\end{equation}
By the Fatou property of $X^s(\T)$ (see \cite{Triebel1983,Triebel1992,RunstSickel1996}) and \eqref{eq:uniform-Xs},
\begin{equation}\label{eq:limit-fatou}
  \|u\|_{L^\infty_TX^s}
  \le \liminf_{n\to\infty}\|u^n\|_{L^\infty_TX^s}
  \le C\|u_0\|_{X^s}.
\end{equation}
We now pass to the nonlinear limit. Set
\[
  q^n=(u^n)^2+\frac12(u_x^n)^2,
  \qquad q=u^2+\frac12u_x^2.
\]
Using \eqref{eq:compact-conv} and the uniform $W^{1,\infty}$ bound from
\eqref{eq:uniform-Xs},
\begin{align*}
  \|q^n-q\|_{C_TL^p}
  &\lesssim \|u^n-u\|_{C_TL^p}
       \|u^n+u\|_{L^\infty_{t,x}}\\
  &\quad+\|u_x^n-u_x\|_{C_TL^p}
       \|u_x^n+u_x\|_{L^\infty_{t,x}}
  \longrightarrow0.
\end{align*}
Since $P_{\T}(D)f=K_{\T}*f$ with $K_{\T}\in L^1(\T)$, Young's inequality yields
\[
  \cN_{\T}(u^n)=K_{\T}*q^n
  \longrightarrow K_{\T}*q=\cN_{\T}(u)
  \quad\text{in }C([0,T];L^p(\T)).
\]
Moreover,
\begin{align*}
  \|u^n\partial_xu^n-u\partial_xu\|_{C_TL^p}
  &\le \|u^n-u\|_{C_TL^p}\|\partial_xu^n\|_{L^\infty_{t,x}}\\
  &\quad+\|u\|_{L^\infty_{t,x}}
       \|\partial_xu^n-\partial_xu\|_{C_TL^p}
  \longrightarrow0.
\end{align*}
Thus one may pass to the limit in the smooth equations in the sense of distributions,
and $u$ solves \eqref{eq:main-CH} with $u(0)=u_0$. By \eqref{eq:limit-fatou},
$u\in L^1(0,T;X^s)$, and Lemma \ref{lem:nonlinear} gives
$\cN_{\T}(u)\in L^1(0,T;X^s)$. Hence Proposition \ref{prop:transport}, applied to
the distributional limit equation with coefficient $a=u$ and source
$g=\cN_{\T}(u)$, yields
\[
  u\in C([0,T];X^s(\T)).
\]
The map $u\mapsto\cN_{\T}(u)$ is locally Lipschitz in $X^s$, while
$u\mapsto u\partial_xu$ is continuous from $X^s$ to $X^{s-1}$. Therefore the
right-hand side of \eqref{eq:main-CH} belongs to $C([0,T];X^{s-1}(\T))$, and
$u_t\in C([0,T];X^{s-1}(\T))$. Thus the asserted solution regularity holds.

\subsection{Uniqueness}
\leavevmode\par
\medskip
Let $u_1,u_2\in C([0,T];X^s)$ be two solutions. Let $y_i$ be the corresponding degree-one Lagrangian lifts and set
\[
  U_i(t,\xi)=u_i(t,y_i(t,\xi)),
  \qquad Y_i(t,\xi)=y_i(t,\xi)-\xi.
\]
Then $U_i,Y_i$ are uniformly bounded in $W^{1,p}\cap W^{1,\infty}$ and 
\begin{equation}\label{eq:Lagrange-system}
  \partial_t y_i=U_i,
  \qquad \partial_t U_i=\widetilde{\cN}_i,
\end{equation}
where
\[
  \widetilde{\cN}_i(\xi)=\cN_{\T}(u_i)(y_i(\xi))=
  \cK_{y_i}A_i(\xi),
\]
with
\[
  A_i=U_i^2(y_i)_\xi+\frac12\frac{(U_i)_\xi^2}{(y_i)_\xi}.
\]
By Lemma \ref{lem:Lagrange-Lipschitz}, for $r=p$ and $r=\infty$,
\begin{equation}\label{eq:Lagrange-N-Lip}
  \|\widetilde{\cN}_1-\widetilde{\cN}_2\|_{W^{1,r}}
  \le C\big(\|U_1-U_2\|_{W^{1,r}}+\|y_1-y_2\|_{W^{1,r}}\big).
\end{equation}
Set 
\begin{equation}\label{eq:D-def}
  D(t)=\|U_1-U_2\|_{W^{1,p}\cap W^{1,\infty}}
       +\|Y_1-Y_2\|_{W^{1,p}\cap W^{1,\infty}}.
\end{equation}
Using \eqref{eq:Lagrange-system} and \eqref{eq:Lagrange-N-Lip},
\[
  D(t)\le D(0)+C\int_0^tD(\tau)\dd\tau.
\]
Gronwall's inequality gives $D(t)\le e^{Ct}D(0)$. If $u_1(0)=u_2(0)$, then $D(0)=0$, hence $U_1=U_2$ and $y_1=y_2$. Since $y_i(t,\cdot)$ is a torus diffeomorphism, $u_1=u_2$.

\subsection{Continuous dependence}
\leavevmode\par
\medskip
Let $u_0^n\to u_0^\infty$ in $X^s(\T)$, and let $u^n,u^\infty$ be the corresponding solutions. From sections 4.1 and 4.2, we have
\begin{equation}\label{eq:CLp-conv}
  u^n\to u^\infty\quad\text{in }C([0,T];L^p(\T)).
\end{equation}
We upgrade this to $X^{s-1}$. Set $f_n=u^n-u^\infty$. For fixed $N$,
\begin{equation}\label{eq:lowfreq-Xsminus1}
  \|S_Nf_n\|_{X^{s-1}}\le C_N\|f_n\|_{L^p}\to0
\end{equation}
in $C([0,T])$. On the other hand,
\begin{equation}\label{eq:highfreq-Xsminus1}
  \|(I-S_N)f_n\|_{X^{s-1}}
  \le C2^{-N}\|f_n\|_{X^s}\le C2^{-N}.
\end{equation}
First choose $N$ large, then $n$ large. Thus
\begin{equation}\label{eq:Xsminus1-conv}
  u^n\to u^\infty\quad\text{in }C([0,T];X^{s-1}(\T)).
\end{equation}

Let
\[
  v^n=u_x^n,
  \qquad v^\infty=u_x^\infty,
  \qquad v_0^n=\partial_xu_0^n,
  \qquad v_0^\infty=\partial_xu_0^\infty.
\]
Differentiating \eqref{eq:main-CH},
\begin{equation}\label{eq:vn-eqn}
  \partial_t v^n+u^n\partial_xv^n=\partial_x\cN_{\T}(u^n)-(v^n)^2.
\end{equation}
Define $w^n$ by
\begin{equation}\label{eq:wn-eqn}
  \left\{
  \begin{aligned}
    &\partial_t w^n+u^n\partial_xw^n=
       \partial_x\cN_{\T}(u^\infty)-(v^\infty)^2,\\
    &w^n(0)=v_0^\infty.
  \end{aligned}\right.
\end{equation}
By Lemma \ref{lem:transport-stability} and \eqref{eq:Xsminus1-conv},
\begin{equation}\label{eq:wn-conv}
  w^n\to v^\infty\quad\text{in }C([0,T];X^{s-1}).
\end{equation}
Set
\[
  z^n=v^n-w^n.
\]
Then
\begin{equation}\label{eq:zn-eqn}
  \partial_t z^n+u^n\partial_x z^n
  =\partial_x(\cN_{\T}(u^n)-\cN_{\T}(u^\infty))-\big((v^n)^2-(v^\infty)^2\big),
\end{equation}
with
\[
  z^n(0)=v_0^n-v_0^\infty.
\]
By (\ref{eq:N-Lipschitz}), the two nonlinear differences in \eqref{eq:zn-eqn} satisfy 
\begin{align}
  \left\|\partial_x\bigl(\cN_{\T}(u^n)-\cN_{\T}(u^\infty)\bigr)\right\|_{X^{s-1}}
  &\le C_R\|u^n-u^\infty\|_{X^s},\label{eq:nonlocal-derivative-difference}\\
  \|(v^n)^2-(v^\infty)^2\|_{X^{s-1}}
  &\le C_R\|v^n-v^\infty\|_{X^{s-1}}.\label{eq:square-difference}
\end{align}
 Together with
\[
  \|u^n-u^\infty\|_{X^s}
  \lesssim \|u^n-u^\infty\|_{X^{s-1}}
  +\|v^n-v^\infty\|_{X^{s-1}},
\]
Proposition \ref{prop:transport} applied to \eqref{eq:zn-eqn} yields
\begin{align}\label{eq:zn-ineq}
  \|z^n(t)\|_{X^{s-1}}
  &\le C\Bigg(\|v_0^n-v_0^\infty\|_{X^{s-1}}
  +\int_0^t\|u^n-u^\infty\|_{X^{s-1}}\dd\tau\notag\\
  &\qquad +\int_0^t\|w^n-v^\infty\|_{X^{s-1}}\dd\tau
  +\int_0^t\|z^n(\tau)\|_{X^{s-1}}\dd\tau\Bigg),
\end{align}
because $v^n-v^\infty=(w^n-v^\infty)+z^n$.
The first three terms on the right of \eqref{eq:zn-ineq} tend to zero, by convergence of the data, \eqref{eq:Xsminus1-conv} and \eqref{eq:wn-conv}. Gronwall's inequality yields
\begin{equation}\label{eq:zn-conv}
  z^n\to0\quad\text{in }C([0,T];X^{s-1}).
\end{equation}
Consequently,
\begin{equation}\label{eq:vn-conv}
  v^n-v^\infty=(w^n-v^\infty)+z^n\to0
  \quad\text{in }C([0,T];X^{s-1}).
\end{equation}
Combining \eqref{eq:Xsminus1-conv} and \eqref{eq:vn-conv} with the norm equivalence
\begin{equation}\label{eq:norm-equivalence}
  \|f\|_{X^s}\lesssim \|f\|_{X^{s-1}}+\|f_x\|_{X^{s-1}},
\end{equation}
we obtain
\[
  u^n\to u^\infty\quad\text{in }C([0,T];X^s(\T)).
\]
Using the equation, the local Lipschitz continuity of $\cN_{\T}$ in $X^s$, and the continuity of
\[
  u\longmapsto u\partial_xu
  \quad\text{from }X^s\text{ to }X^{s-1},
\]
we further obtain
\begin{align*}
  \|\partial_tu^n-\partial_tu^\infty\|_{C_TX^{s-1}}
  &\le \|u^n\partial_xu^n-u^\infty\partial_xu^\infty\|_{C_TX^{s-1}}\\
  &\quad+\|\cN_{\T}(u^n)-\cN_{\T}(u^\infty)\|_{C_TX^{s-1}}
  \longrightarrow0.
\end{align*}
Hence the data-to-solution map is continuous into
$C([0,T];X^s)\cap C^1([0,T];X^{s-1})$.
This completes the proof of Theorem \ref{thm:main}.

\section{Norm inflation in the periodic critical Triebel--Lizorkin spaces}\label{sec:norm-inflation}

In this section we prove Theorem \ref{thm:norm-inflation}. The construction is the periodic realization of the smooth atomic mechanism introduced in \cite{ZhangYan2026Sharp}. The wave-breaking and logarithmic continuation arguments are the periodic ones used in \cite{GuoLiuMolinetYin2019}.

\subsection{A nested periodic atomic construction}
\leavevmode\par
\medskip
Fix an even function
\[
  \eta\in C_c^\infty((-1/16,1/16)),
  \qquad \eta(0)=1.
\]
For $k\ge 0$, define 
\begin{equation}\label{eq:periodic-atom-def}
 \widetilde a_k(x):=-x\eta(2^kx),
 \qquad x\in\R,
\end{equation}
and let $a_k=\mathcal P\widetilde{a_k}(x)$. Since $\widetilde{a_k}$ vanishes in a neighbourhood of the endpoints, the extension $a_k$ belongs to $C^\infty(\T)$. Moreover, $a_k$ is odd,
\begin{equation}\label{eq:atom-derivative-origin}
  a_k'(0)=-1,
\end{equation}
and
\[
  \operatorname{supp}a_k\subset Q_k
  :=\{x\in\T:\operatorname{dist}_{\T}(x,0)\le 2^{-k}/16\}.
\]
For every integer $\ell\ge0$,
\begin{equation}\label{eq:ak-derivative-bound}
  \|\partial_x^\ell a_k\|_{L^\infty}
  \le C_\ell2^{-k+k\ell}.
\end{equation}
Thus the family $a_k$ satisfies the hypotheses of Lemma \ref{lem:periodic-atoms}.

For $N\ge 0$, define
\begin{equation}\label{eq:fN-def}
  f_N(x)=\sum_{k=1}^N\frac1k a_k(x).
\end{equation}
Then $f_N$ is a smooth real-valued odd periodic function and
\begin{equation}\label{eq:fN-derivative}
  f_N'(0)=-\sum_{k=1}^N\frac1k\longrightarrow-\infty.
\end{equation}

\begin{lemma}\label{lem:fN-uniform}
Let $1<p<\infty$, $1\le q\le\infty$, and $s=1+1/p$. There exists $C=C(p,q,\eta)$ such that
\begin{equation}\label{eq:fN-uniform}
  \sup_{N\ge 0}\|f_N\|_{F^s_{p,q}(\T)}\le C.
\end{equation}
\end{lemma}

\begin{proof}
By Lemma \ref{lem:periodic-atoms},
\begin{equation}\label{eq:fN-sequence-bound}
  \|f_N\|_{F^s_{p,q}}
  \le C\|\Lambda_N\|_{L^p},
\end{equation}
where
\[
  \Lambda_N(x)
  =\left(\sum_{k=1}^N
  \left(\frac{2^{k/p}}{k}\mathbf1_{Q_k}(x)\right)^q
  \right)^{1/q},
\]
with the usual supremum when $q=\infty$. Put
\[
  A_j=Q_j\setminus Q_{j+1},
  \qquad 0\le j<N.
\]
If $x\in A_j$, then $x\in Q_k$ exactly for $k\le j$. Since $2^{k/p}/k$ grows geometrically for large $k$,
\begin{equation}\label{eq:geometric-sequence}
  \left(\sum_{k=1}^j
  \left(\frac{2^{k/p}}k\right)^q\right)^{1/q}
  \le C\frac{2^{j/p}}j,
\end{equation}
again with the evident modification for $q=\infty$. Since $|A_j|\asymp2^{-j}$ and $|Q_N|\asymp2^{-N}$,
\begin{align*}
  \|\Lambda_N\|_{L^p}^p
  &\le C\sum_{j=1}^{N-1}
     2^{-j}\left(\frac{2^{j/p}}j\right)^p
     +C2^{-N}\left(\frac{2^{N/p}}N\right)^p\\
  &\le C\sum_{j=1}^\infty\frac1{j^p}<\infty.
\end{align*}
The last series converges precisely because $p>1$. This proves \eqref{eq:fN-uniform}.
\end{proof}

\medskip
\noindent\textit{Choice of the initial data.}
Let $C_*=\sup_N\|f_N\|_{F^s_{p,q}}$ and set
\[
  \alpha_\varepsilon=\frac{\varepsilon}{2C_*}.
\]
By \eqref{eq:fN-derivative}, one may choose $N=N(\varepsilon)$ so large that
\[
  \sum_{k=1}^{N}\frac1k>
  \frac{4C_*}{\varepsilon^2}.
\]
Define
\begin{equation}\label{eq:u0eps-def}
  u_{0,\varepsilon}=\alpha_\varepsilon f_{N(\varepsilon)}.
\end{equation}
Then
\begin{equation}\label{eq:u0eps-properties}
  \|u_{0,\varepsilon}\|_{F^s_{p,q}}
  \le\frac\varepsilon2,
  \qquad
  u_{0,\varepsilon}'(0)<-\frac2\varepsilon.
\end{equation}
The function $u_{0,\varepsilon}$ is smooth, real-valued and odd.

\subsection{Blow-up of the critical Triebel--Lizorkin norm}
\leavevmode\par
\medskip
We first recall the periodic wave-breaking estimate and Sobolev blow-up alternative needed here. The proof in \cite{GuoLiuMolinetYin2019} is already formulated for both $\R$ and $\T$. 

\begin{lemma}[Odd-data lifespan bound]\label{lem:odd-blowup}
Let $u_0\in H^3(\T)$ be real-valued and odd, with $u_0'(0)<0$. Then the corresponding smooth periodic solution has a finite maximal lifespan $T$ satisfying
\begin{equation}\label{eq:lifespan-slope}
  T\le\frac{2}{|u_0'(0)|}.
\end{equation}
\end{lemma}
Applying Lemma \ref{lem:odd-blowup} to \eqref{eq:u0eps-properties}, the maximal lifespan $T_\varepsilon$ satisfies
\begin{equation}\label{eq:Teps-small}
  T_\varepsilon\le\frac2{|u_{0,\varepsilon}'(0)|}<\varepsilon.
\end{equation}
\begin{lemma} \label{5.3}
 Given $u_0(x) \in H^s, s>3 / 2$, there exists a maximal $T \geq \tilde{T}\left(\left\|u_0\right\|_{H^{\frac{3}{2}+}}\right)>0$ and a unique solution $u$ to (\ref{eq:main-CH}) such that
	
	$$
	u \in C\left([0, T) ; H^s\right) \cap C^1\left([0, T) ; H^{s-1}\right) .
	$$
	Moreover, the solution depends continuously on the initial data, i.e. for any $T^{\prime}<T$ the mapping $u_0 \mapsto u: H^s \mapsto C\left(\left[0, T^{\prime}\right] ; H^s\right) \cap C^1\left(\left[0, T^{\prime}\right] ; H^{s-1}\right)$ is continuous on a $H^s$-neighborhood of $u_0$, and if $T<\infty$, then $\lim _{t \rightarrow T^{-}}\|u(t)\|_{H^s}=\infty$.
\end{lemma}
We recall the periodic logarithmic estimate.

\begin{lemma}[\text{\cite[Lemma 3.2]{GuoLiuMolinetYin2019}}]\label{lem:log-estimate}
For every $v\in H^2(\T)$,
\begin{equation}\label{eq:periodic-log-estimate}
  \|v_x\|_{L^\infty}
  \le C\left(1+\|v\|_{B^1_{\infty,\infty}}
  \log_{2}\big(2+\|v\|^{2}_{H^2}\big)\right).
\end{equation}
\end{lemma}
  We now finish the proof of Theorem \ref{thm:norm-inflation}. If
\[
  \sup_{0\le t<T_\varepsilon}
  \|u_\varepsilon(t)\|_{B^1_{\infty,\infty}}<\infty,
\]
then Lemma \ref{lem:log-estimate}, together with the standard energy inequality
\[
  \frac{d}{dt}\|u_\varepsilon(t)\|_{H^2}^2
  \lesssim \|u_{\varepsilon,x}(t)\|_{L^\infty}
  \|u_\varepsilon(t)\|_{H^2}^2,
\]
and the Gronwall inequality gives
\[
  \sup_{0\le t<T_\varepsilon}\|u_\varepsilon(t)\|_{H^2}<\infty.
\]
This contradicts the blow-up criteria in Lemma \ref{5.3}. Hence
\begin{equation}\label{eq:Binf-blowup}
  \limsup_{t\uparrow T_\varepsilon}
  \|u_\varepsilon(t)\|_{B^1_{\infty,\infty}}=\infty.
\end{equation}
The critical embedding
\begin{equation}\label{eq:critical-F-embedding}
  F^{1+1/p}_{p,q}(\T)\hookrightarrow B^1_{\infty,\infty}(\T)
\end{equation}
implies that \eqref{eq:Binf-blowup} forces the $F^{1+1/p}_{p,q}$ norm to be unbounded as well. This proves both assertions in \eqref{eq:norm-inflation-main}.

\end{document}